\documentclass[11pt,twoside]{article}

\usepackage{mathrsfs,amsfonts,amsmath,amssymb}
\usepackage{latexsym}
\newtheorem{satz}{Theorem}

\newtheorem{theorem}[satz]{Theorem}
\newtheorem{lemma}[satz]{Lemma}

\def\no{\noindent}
\def\sbeq{\subseteq}

\def\N{\mathbb {N}}
\def\Z{\mathbb {Z}}
\def\F{\mathbb {F}}

\def\zn{\Z/N\Z}
\def\e{\varepsilon}

\def\s{\sigma}

\def\a{\alpha}

\def\C{\mathbb{C}}

\def\h{\widehat}

\def\d{\delta}

\def\({\big (}
\def\){\big )}
\def\g{\gamma}
\def\G{\Gamma}

\def\b{\beta}

\def\dim{{\rm dim}}
\def\le{\leqslant}
\def\ge{\geqslant}
\def\_phi{\varphi}

\def\m{\times}

\def\k{\kappa}

\def\D{\Delta}
\def\s{\sigma}
\def\t{\tau}
\def\th{\theta}

\def\th{\theta}

\def\rk{{\rm{ rk}}}
\def\m{\mu}

\def\Span{{\rm Span\,}}

\def\D{\Delta}

\def\La{\Lambda}
\def\F{\mathbb {F}}
\def\hA{\h {1_A}}
\def\hB{\h {1_B}}

\begin{document}
\title{\bf Improved bound in  Roth's theorem on arithmetic progressions}

\author{ By\\  \\{\sc Tomasz Schoen\footnote{The author is partially supported by National Science Centre, Poland grant 2019/35/B/ST1/00264}}}

\date{}

\maketitle

\begin{abstract} We prove that if $A\sbeq \{1,\dots,N\}$ does not contain any non-trivial three-term arithmetic progression, then
$$|A|\ll \frac{(\log\log N)^{3+o(1)}}{\log N^{}}N\,.$$ 
\end{abstract}

\section{Introduction}\label{s:intro}

In this paper we prove the following bound in Roth's theorem on arithmetic progressions.

\begin{theorem}\label{t:roth} If $A\sbeq \{1,\dots, N\}$ does not contain any non-trivial  arithmetic progression of length three then
$$|A|\ll \frac{(\log\log N)^3(\log\log\log N)^5}{\log N^{}}N\,.$$
\end{theorem}

The first non-trivial upper bound concerning the size of progression-free
sets was given by Roth \cite{roth} who showed the above inequality with $N/\log\log N$. Then it was subsequently refined
by Heath-Brown \cite{heath-brown} and Szemer\'edi \cite{szemeredi-3ap} with a denominator of $(\log N)^c$ for a positive constant $c$,
 by  Bourgain \cite{bourgain-1/2, bourgain-2/3} and 
Sanders \cite{sanders-3/4}
  by proving that bound with $c=1/2-o(1)$, $c=2/3-o(1)$ and $c=3/4-o(1)$.  Sanders \cite{sanders-1} showed a result close to the logarithmic barrier
	$$|A|\ll \frac{(\log\log N)^{6}}{\log N^{}}N$$
	and  Bloom \cite{bloom} further proved that
		$$|A|\ll \frac{(\log\log N)^{4}}{\log N^{}}N\,,$$
		for  set $A\sbeq \{1,\dots, N\}$  avoiding three-term arithmetic progressions. Recently a slightly weaker bound was obtained by a different argument by Bloom and Sisask \cite{bloom-sisask}. Other results related to Roth's theorem can be found in  
\cite{green-roth}, \cite{helfgott-deroton}, \cite{naslund}, \cite{schoen-shkredov} and \cite{schoen-sisask}. 

Let us also comment on the recent progress for the analogous problem
in  a high-dimensional case. Croot, Lev and Pach
\cite{croot-lev-pach} proved, by a polynomial method, an upper
estimate $(4-c)^n$ with some constant $c>0$, for the size of
progression-free sets in $(\Z/4\Z)^n$. Later Ellenberg and Gijswijt
\cite{ellenberg-gijswijt} obtained the bound  $(3-c)^n$ with a
positive constant $c$  for subsets of $\F_3^n$. The latter result significantly
improves  the previous best  bound of Bateman and Katz
\cite{bateman-katz},  however this the paper
\cite{bateman-katz} contains many deep results and valuable ideas
that could  potentially be also used in the integer case.

Each of the  mentioned papers contains significant novel ideas
and methods, any of them are used in our proof of Theorem
\ref{t:roth}. We employ the density increment argument obtained via
the Fourier analytical method invented by Roth \cite{roth}. We  make
use of the Bohr set machinery introduced by Bourgain
\cite{bourgain-1/2}. We focus on the structure of the large
spectrum, explored  first by Bourgain \cite{bourgain-2/3} and
thenceforth used in all further works. We also take advantage of 
deep insight into the structure of the large spectrum done by
Bateman and Katz in \cite{bateman-katz} and
\cite{bateman-katz-nonsmooth}.

Finally, let us mention that as far as the lower bound on the maximal size of progression-free subsets of $\{1,\dots, N\}$ is concerned,
the first non-trivial lower estimate $N^{1-c(\log \log N)^{-1}}$ was established by Salem and Spencer \cite{salem-spencer}. Then Behrend \cite{behrend} improved it to
$\exp(-c\sqrt{\log N})N.$ Elkin \cite{elkin}  refined slightly Behrend's bound by a factor of $(\log N)^{1/2}$ and his argument was simplified in \cite{green-wolf}.

\section{Notation, Bohr sets and standard results} \label{s:notation}

All sets considered in the paper are finite subsets of $\Z$ or $\Z/N\Z$.
We write $1_A(x)$ for the indicator function of  set $A$. Given functions
$f,g:\Z/N\Z\rightarrow\mathbb{C}$, the convolution of $f$ and $g$
is defined by
$$
(f*g)(x)=\sum_{t\in\Z/N\Z}f(t)g(x-t).
$$
The Fourier coefficients of a function $f:\Z/N\Z\to\C$ are defined
by
$$
\h f(r)=\sum_{x\in\Z/N\Z}f(x)e^{-2\pi ixr/N},
$$
where $r\in\Z/N\Z$, and the above applies to the indicator function
of  $A\sbeq\Z/N\Z$ as well.
Parseval's formula states in particular that 
$$\sum_{r=0}^{N-1}|\hA(r)|^{2}=|A|N\,.$$
We also recall the fact  that 
$$\h{(1_A*1_B)}(r)=\hA(r)\hB(r)\,.$$
For a real number $\th\ge 0$,  the $\th-$spectrum of a set $A$ is the set
$$
\D_{\th}(A)=\big \{r\in\Z/N\Z:|\hA(r)|\ge\th|A|\big \}.
$$
For a specified set $A$ we often write $\D_{\th}$ instead of $\D_{\th}(A).$ 

For $m\in \N$ by  $E_{2m}(A)$  we denote  the number of $2m$--tuples $(a_1,\dots,a_m,b_1,\dots,b_m)\in A^{2m}$ such that
$$a_1+\dots+a_m=b_1+\dots+b_m.$$ For $m=2$, we simply write $E(A)$ for $E_4(A)$ and we call it  the additive energy of  set $A.$

We define the span of a finite set $X$ by
$$\Span(X)=\Big\{\sum_{x\in X}\e_xx: \e_x\in \{-1,0,1\} \text{ for all } x\in X\Big\}.$$
The dimension $\dim(A)$ of  set $A$ is the minimal size of  set $X$ such that $A\sbeq \Span (X).$ The
following theorem proven in \cite{sanders-shkredov} (see also \cite{schoen-shkredov-dim} and \cite{sy}) provides an upper bound on the dimension of a set in terms of its additive doubling $K=|A+A|/|A|.$

\begin{theorem}\label{t:dimension}$\text{\cite{sanders-shkredov}}$ Suppose that  $|A+A|= K|A|$. Then $\dim (A)\ll K\log |A|.$
\end{theorem}

We are going to use a  sophisticated
concept of Bohr sets-a fundamental tool introduced to modern additive combinatorics by Bourgain \cite{bourgain-1/2}.

Let $G=\Z/N\Z$ be a cyclic group and let us denote the group
of its characters by $\h G\backsimeq\Z/N\Z$. We define the Bohr set
with a generating set $\G\sbeq\h G$ and a radius  $\g\in(0,\frac{1}{2}]$
to be the set
$$
B(\G,\g)
  =  \big\{ x\in \Z/N\Z: \ \|{tx}/{N}\|\le\g \text { for all } t\in\G\big\}\,.$$
 Here $\left\Vert \cdot\right\Vert $
denotes the distance to the integers, i.e. $\left\Vert x\right\Vert =\min_{y\in\Z}|x-y|$
for $x\in\mathbb{R}$. Given $\eta>0$ and a
Bohr set $B=B(\G,\g)$, by $B_{\eta}$ we mean the Bohr set $B(\G,\eta\g).$
The two lemmas below
are pretty standard, hence we refer the reader to \cite{tao-vu}
for a  complete account. The size of $\G$ is called the rank of $B$ and we denote it by $\rk(B).$
\begin{lemma}\label{l:bohr-size}
 For every $\g\in(0,\frac{1}{2}]$ we have
$$
\g^{|\G|}N\le|B(\G,\g)|\le 8^{|\G|+1}|B(\G,\g/2)|\,.
$$
\end{lemma}

Bohr sets do not always behave like convex bodies. The size of Bohr
sets can vary significantly even for small changes of the radius
which was the motivation behind  the following definition.

We call a Bohr set $B(\G,\g)$ \emph{regular} if for every $\eta$,
with $|\eta|\le1/(100|\G|)$ we have
$$
(1-100|\G||\eta|)|B|\le|B_{1+\eta}|\le(1+100|\G||\eta|)|B|.
$$
Bourgain \cite{bourgain-1/2} showed that regular Bohr sets are ubiquitous.
\begin{lemma}\label{l:bohr-regular}
\label{lem:regularity} For every Bohr set $B(\G,\g),$ there exists
$\g'$ such that $\frac{1}{2}\g\le\g'\le\g$ and $B(\G,\g')$ is
regular.
\end{lemma}

The last lemma of this section presents a standard $L^2$ density increment technique introduced by Heath-Brown \cite{heath-brown} and Szemer\'edi \cite{szemeredi-3ap}, see also \cite{pintz-steiger-szemeredi}.
A proof of the lemma below can be found in either of the following papers \cite{sanders-3/4}, \cite{sanders-1} and \cite{bloom}.

\begin{lemma}
\label{l:l2-increment} Let $A\sbeq \Z/N\Z$ be a set with density
$\d.$ Let $\G\sbeq \zn$ and $\nu\ge 0$ be such that
$$\sum_{r\in \G\setminus \{0\}}|\hA(r)|^{2}\ge \nu|A|^{2}\,.$$
Then there is a regular Bohr set $B$  with  $\rk
({B})=\dim(\G)$ and radius $\Omega((\dim(\G))^{-1})$ such that
$$|(A+t)\cap B|\ge(1+\Omega(\nu))\d |B|$$
 for some $t$.
\end{lemma}

Throughout the paper we assume that  set $A$ does not contain any non-trivial arithmetic progression of length three and that $N$ is a large number.

\section
{Sketch of the argument}\label{s:sketch}

We apply a widely used  density increment argument introduced by
Roth \cite{roth}, however we use it  in a rather non-standard way.
In the first step, we increase the density by a  large factor of the form
$(\log(1/\d))^{1-o(1)}$  on some low-rank Bohr
set.  Then we apply the iterative method of Bloom  to our
new set with larger density  to obtain the desired bound.

The general strategy can be roughly described as follows. Let $A\sbeq [N]$ be a set with density $\d$ without arithmetic progressions of length three then it is known that $1_A$ must have large  Fourier coefficients. To obtain a density increment  we would have to find a small set $\La$ such that   $\Span (\La)$ has  large intersection with the spectrum $\D_\d.$ The size of $\La$ is equal to the rank of a Bohr set, on which we will increase  density, and   density increment
(given by the $L^2$ method) equals
$$(1+\Omega(\d^2|\Span(\La)\cap \D_\d|))\d.$$
If we want to obtain the density increment by  factor $\Omega(L)$ for some function $L\rightarrow \infty$, we have to locate
 $\La$ of size $O(\d^{-1+c}), c>0$ such that
\begin{equation}\label{span-spec}|\Span(\La)\cap \D_\d|
\gg L\d^{-2}\,.
\end{equation}
The main problem is that by of Bateman-Katz structural result (see Theorem \ref{t:bateman-katz-structure}) there are sets with spectrum such that the described set $\La$ does not exist.
Hence one needs to combine the above method with some  new ideas.

In order to obtain the density increment we will consider three separate cases with respect to the size of Fourier coefficients of $1_A$ that in a sense dominate in $\D_{\d^{1+\m}}$ for a some small constant $\m>0.$ If 
the contribution of middle size or small Fourier coefficients is large we follow  the method introduced by Bateman and Katz \cite{bateman-katz}. We consider essentially two subcases
 according to the additive behavior  of the large spectrum $\D$. Following \cite{bateman-katz}, we call the cases smoothing and nonsmoothing, respectively. If the higher energy $E_8(\D)$ is much bigger than
 one can deduce from the H\"older inequality applied to $E(\D)$ (smoothing case), then based on the Bateman-Katz argument we can indeed find a small set $\La$ satisfying (\ref{span-spec}). The nonsmoothing case is more delicate.
 In that case we use a seminal result of Bateman and Katz
\cite{bateman-katz,bateman-katz-nonsmooth} that describes  the
structure of the spectrum in the nonsmoothing case and it turns out
that again we can also find a small set $\La$ satisfying
(\ref{span-spec}), apart from one situation where roughly $\D\approx
X+H\,,|\D|\sim\d^{-3+O(\m)}\,, |X|\sim \d^{-2+O(\m)}\,, |H|\sim \d^{-1+O(\m)}$ and $H$
is a highly structured set. This case is considered separately in
Lemma \ref{l:increment-3-hard} which is an important  part of our
argument. We  show that either the density can  be increased  on a Bohr
set generated by $H$ (such a Bohr set has a very low rank) or $X$
contains additive substructure which again leads to a density
increment on a low-rank Bohr set.
The above argument does not apply when $\D_{\d^{1+\m}}$ is dominated by large  Fourier coefficients. Then assuming that there are very few smaller  Fourier coefficients in $\D_{\d^{1+\m}}$, using a different technique based on Fourier  approximation method, we prove that $A$ does indeed have density increment on a low-rank Bohr set.

\section{Middle size Fourier coefficients}
\label{s:large-coeff}
Assume that $A\sbeq \{1,\dots,N'\}$ does not contain any non-trivial arithmetic progressions of length three. Let $N$ be any prime number satisfying $2N'<N\le 4N'.$ We embed $A$ in $\zn$ in a natural way and observe that $A$ also does not  contain any non-trivial arithmetic progression of length $3$ in $\zn.$ Let us recall a standard argument that shows that $1_A$ must have large Fourier coefficients. The number of three-term arithmetic progressions in $A$ (including trivial ones) is expressed by the sum $\frac1N\sum_{r=0}^{N-1}\hA(r)^2\hA(-2r)$, whence we have
$$\frac1N\sum_{r=0}^{N-1}\hA(r)^2\hA(-2r)=|A|\,.$$
Clearly, we can assume that $|A|>\sqrt{2 N}$, so  by the H\"older inequality
 $$\sum_{r\not=0}|\hA(r)|^3\ge |A|^3-N|A|\ge \frac12|A|^3\,.$$
Since
$$\sum_{r\not \in \D_{\d/4}(A)}|\hA(r)|^3\le \frac14\d |A|\sum_{r=0}^{N-1}|\hA(r)|^2=\frac14|A|^3$$
it follows that
$$\sum_{r \in \D_{\d/4}\setminus \{0\}}|\hA(r)|^3\ge \frac14|A|^3\,,$$
hence there are non-trivial Fourier coefficients with $|\hA(r)|\gg \d|A|.$

However to obtain a large density increment we have to control Fourier coefficients below typical treshold $\d|A|.$
We will consider three separate cases: 
\begin{equation}\label{mid}
\sum_{r:\, \d^{1-\m} |A|\le |\hA(r)|\le \d^{1/10} |A|}|\hA(r)|^3\ge \frac1{10}\d^{\m/5}|A|^3\,,
\end{equation}
\begin{equation}\label{sml}
\sum_{r:\, \d^{1+\m} |A|\le |\hA(r)|\le \d^{1-\m} |A|}|\hA(r)|^3\ge \frac1{10}\d^{\m/5}|A|^3\,,
\end{equation}
and the last one if \eqref{mid} and \eqref{sml} do not hold, where $\m$ is a small positive constant. Throughout the paper we assume that $\d^{\m/20}<\log^{-1}(1/\d)$
since we know that $\d\rightarrow 0$ as $N\rightarrow\infty.$

By dyadic argument, we obtain
\begin{equation}\label{in:l3-fourier}
\sum_{r:\, \theta |A|\le |\hA(r)|\le 2\theta |A|}|\hA(r)|^3\ge  \frac1{10}\d^{\m/5}|A|^3\log^{-1}(1/\d)
\end{equation}
for some $\d^{1-\m}\le \theta \le \d^{1/10}$, so 
\begin{equation}\label{in:spectrum-size}
|\D_\th|\gg \theta^{-3}\d^{\m/5}\log^{-1}(1/\d)\ge \theta^{-3}\d^{\m/4}\,.
\end{equation}

In this section we consider the first case \eqref{mid}.
We will apply  the Bateman-Katz-Bloom lemma, see Lemma 5.3 in \cite{bateman-katz} and  Theorem 4.1 in \cite{bloom} (a slightly weaker version of Lemma \ref{l:bloom} can be easily deduced from Lemma \ref{l:bateman-katz-span}). 

\begin{lemma}\label{l:bloom}
Let $A\sbeq \Z/N\Z$ be a set with density $\d,$ and let $\D$ be a subset of  $\D_\th$. Then there exists a set $\D'\sbeq \D$ such that $|\D'|\gg \th |\D|$ and
$\dim(\D')\ll\th^{-1}\log(1/\d).$
  \end{lemma}

\begin{lemma}\label{l:increment-1}
     Let $A\subseteq \Z/N\Z$ be a set with density $\d,$  and suppose that (\ref{in:spectrum-size})  holds for some
        $\d^{1-\m}\le \theta \le \d^{1/10}$.  Then  there is a regular Bohr set $B$ with $\rk(B)\ll \d^{-1+\m/3}$  and  radius $\Omega( \d^{1-\m/3})$ such that for some $t$
    $$|(A+t)\cap B|\gg \d^{1-\m/4}|B|.$$
\end{lemma}
\begin{proof}
By Lemma \ref{l:bloom} there exists a set
$\D_1\sbeq \D_\th$ such that $|\D_1|=\Theta (\th |\D_\th|)$ and   
$$\dim(\D_1)\ll\theta^{-1}\log(1/\d)\,.$$
By iterative
application of Lemma \ref{l:bloom},
 we see that there are
disjoint sets   $\D_1,\dots, \D_k\sbeq \D_\th$, for
$k=\Theta(\d^{-\m/2})$ such that $|\D_i|=\Theta(\theta
|\D_\th|)$  and 
$$\dim (\D_i)\ll \theta^{-1}\log(1/\d)$$
 for
every $1\le i\le k.$ Put $\G=\bigcup_{i=1}^k \D_i\sbeq \D_\th$ then
by (\ref{in:spectrum-size}) we have
$$|\G|\gg\d^{-\m/2}\theta^{-2}\d^{\m/4}\gg \d^{-\m/4}\theta^{-2}$$ and
$$\dim(\G)\ll\d^{-\m/2}\th^{-1}\log(1/\d)\ll \d^{-1+\m/2}\log(1/\d)\ll \d^{-1+\m/3}\,.$$
Therefore, by Lemma
\ref{l:l2-increment} a shift of the set $A$ has density at least
$$(1+\Omega(\th^2|\G|))\d\gg \d^{-1+\m/3}$$
on a regular  Bohr
 set with  rank $O(\d^{-1+\m/3})$ and   radius $\Omega( \d^{1-\m/3})$.$\hfill\Box$
\end{proof}

\bigskip

\section{Additively smoothing spectrum}\label{s:smoothig}

In sections \ref{s:smoothig} and \ref{s:nonsmoothing} we obtain a density increment provided that \eqref{sml} holds. Hence for some $\d^{1+\m}\le \th\le \d^{1-\m}$ we have
$$\sum_{r:\, \th |A|\le |\hA(r)|\le 2\th |A|}|\hA(r)|^3\gg |A|^3\log^{-1}(1/\d)\,.$$
Thus
\begin{equation}\label{in:spectrum-size-1}
|\D_{\th}|\ge  \d^{\m/5}\th^{-3}\log^{-1} (1/\d)\ge \d^{2\m}\th^{-2}\d^{-1}\,,
\end{equation}
so the size of $\D_{\th}$ is close to the maximal possible value.

A well-known theorem of Shkredov \cite{shkredov-1, shkredov-2} states that for every $\D\sbeq \D_\th$ and $m\in \N$ we have
$$E_{2m}(\D)\ge \th^{2m}\d|\D|^{2m}.$$
Observe that by the Parseval formula $|\D_\th|\le \th^{-2}\d^{-1}$,
so if we additionally assume that $|\D_\th|\gg \th^{-2}\d^{-1},$
then the H\"older inequality implies that for $m\ge 3$
$$E_{2m}(\D_\th)\ge E(\D_\th)^{m-1} |\D_\th|^{m-2}\gg_m \th^{2m}\d|\D_\th|^{2m}\,,$$
which  essentially meets Shkredov's bound.
This observation  motivates  the next definition introduced by Bateman and Katz. We say that a spectrum $\D_\th$ is $\s$-additively smoothing (or simply additively smoothing if $\s$ is indicated) if
$$E_{8}(\D_\th)\ge \d^{-\s}\th^{8}\d^{}|\D_\th|^{8}.$$
Otherwise, we say that the spectrum $\D_\th$ is $\s$-additively nonsmoothing. In this section, we will obtain a density increment for  additively smoothing spectrum.

The following lemma, proven in \cite{schoen-shkredov-dim} (see Corollary 7.5)  is an abelian group version of  Bateman-Katz Lemma 5.3. The proof of this result requires some modifications,
but similarly as in Bloom's Theorem 4.1 in \cite{bloom} it  relies on a probabilistic argument of Bateman and Katz.

\begin{lemma}\label{l:bateman-katz-span}
    Let $\D\subseteq \zn$ be a  set such that $E_{2s}(\D)=\k |\D|^{2s}\ge 10^s s^{2s}|\D|^s,$ where  $2\le s=\lfloor \log |\D|\rfloor$.  Then there exists a set $\La\sbeq \D$ such that
  $|\La|\ll\k^{-1/2s}\log^{3/2} |\D|$ and
    \begin{equation*}\label{f:d_bound1}
     |\Span (\La)\cap \D|\gg \k^{1/2s}|\D|\log ^{-3/2}|\D|\,.
    \end{equation*}
\label{l:bateman-katz}
\end{lemma}

\begin{lemma}\label{l:increment-2}
    Let $A\subseteq \Z/N\Z$,  $|A|=\d N$  and suppose
        that for some $\d^{1+\m}\le \th\le \d^{1-\m}$ we have $E_8(\D_\th)\ge \d^{-20\m}\th^{8}\d^{}|\D_\th|^{8}.$ Then  there is a regular Bohr set $B$ with rank  $\rk(B)\ll \d^{-1+\m/2}$ and radius
    $\Omega(\d^{1-\m/2})$ such that  for some $t$
    $$|(A+t)\cap B|\gg \d^{1-\m/2}|B|.$$
\end{lemma}
\begin{proof} Put  $s=\lfloor \log |\D_\th|\rfloor$.
Using the H\"older inequality and (\ref{in:spectrum-size}) we have
\begin{eqnarray*}
E_{2s}(\D_\th)&\ge& E_8(\D_\th)^{\frac{s-1}{3}}|\D_\th|^{-\frac{s-4}{3}}\ge \d^{O(1)}\d^{\frac13(1-20\m)s}\th^{\frac83s}|\D_\th|^{\frac73s}\\
&\ge& \d^{-\frac13(18\m+o(1))s} \th^{2s}|\D_\th|^{2s}\ge \d^{-4\m s} \th^{2s}|\D_\th|^{2s}\,,
\end{eqnarray*}
provided that $N$ is large enough. Notice that (\ref{in:spectrum-size-1}) implies  that $E_{2s}(\D_\th)\gg 10^{s} s^{2s}
|\D_\th|^s,$  so we can apply
 Lemma \ref{l:bateman-katz-span}. Thus, there exists a set $\La$ such that
$$|\La|\ll \d^{2\m}\th^{-1}\log^{3/2}(1/\d)\ll \d^{-1+\m/2}$$
and
$$ |\Span (\La)\cap \D_\th|\gg \d^{-2\m}\th \log^{-3/2}(1/\d)|\D_\th|\gg \d^{-2\m}\th^{-2}\log^{-5/2}(1/\d)\gg \d^{-\m/2}\th^{-2}\,.$$
Now it is enough to use  Lemma \ref{l:l2-increment} with $\G=\Span(\La)$ to get the required result.$\hfill\Box$

\end{proof}

\section{Additively nonsmoothing spectrum}
\label{s:nonsmoothing}

In this section, we will obtain a density increment in a more
difficult case, when the spectrum $\D_\th$ is an additively
nonsmoothing set. Recall that  for some $\d^{1+\m}\le \th\le \d^{1-\m}$ we have
$$
|\D_{\th}|\ge \d^{2\m}\th^{-2}\d^{-1}\,.
$$
 Bateman and Katz
\cite{bateman-katz,bateman-katz-nonsmooth} proved the following
fundamental result characterizing  the structure of additively
nonsmoothing sets.

\begin{theorem}\label{t:bateman-katz-structure}
Let $\tau>0$ be a fixed number. There exists a function $f=f_\tau:(0,1)\rightarrow (0,\infty)$ with $f(x)\rightarrow 0$ as $x\rightarrow 0$ such that the following holds. Let $\D$ be a symmetric set of an abelian group and let $\s>0$. Assume that $E(\D')\gg |\D|^{2+\tau}$ for every $\D'\sbeq \D$ with $|\D'|\gg|\D|$ and that $E_8(\D)\le |\D|^{4+3\t+\s}.$ Then there exists $\a$, $0\le \a\le \frac{1-\t}{2}$, such that for $i=1,\dots,\lceil |\D|^{\a-f(\s)}\rceil$ there are sets $H_i, X_i$ and $\D_i\sbeq \D$ such that
\begin{eqnarray}\label{bk1} |H_i|&\ll& |\D|^{\t+\a+f(\s)},\\
\label{bk2}|X_i|&\ll&  |\D|^{1-\t-2\a+f(\s)},\\
\label{bk3}|H_i+H_i|&\ll& |H_i|^{1+f(\s)},
\end{eqnarray}
and
\begin{eqnarray}\label{bk4}|(X_i+H_i)\cap \D_i|&\gg&  |\D|^{1-\a-f(\s)}.\end{eqnarray}
Furthermore, the sets $\D_i$ are pairwise disjoint.
\end{theorem}

We will apply Theorem \ref{t:bateman-katz-structure} to the set $\D_\th.$ By Shkredov's theorem  for every $\D\sbeq \D_\th$ with $|\D|\gg |\D_\th|$
we have
$$E(\D)\ge \th^4\d |\D|^4\gg \d^{10\m/3}\th^{2/3}\d^{-2/3}|\D|^{7/3}\ge \d^{4\m}|\D|^{7/3}\gg |\D_\th|^{7/3-2\m}\,.$$
On the other hand, by Lemma \ref{l:increment-2} we can assume that
$$E_8(\D_\th)\le \d^{-20\m}\th^{8}\d^{}|\D_\th|^{8}\le |\D_\th|^{5+10\m}\,.$$
Therefore we can apply Theorem \ref{t:bateman-katz-structure} with
$$ \t=1/3-2\m \text{~~ and ~~}  \s=16\m\,,$$
hence our spectrum $\D_\th$ has structure described in Theorem \ref{t:bateman-katz-structure}.

Throughout the paper assume that $f(\s)\ge \s$  and that $\m$ and $f=f(16\m)$
are small constants.

Each of the four inequalities given in Theorem
\ref{t:bateman-katz-structure}   is crucial in our approach. Note that from \eqref{bk1}, \eqref{bk2} and \eqref{bk4} we can deduce  lower bounds for the size of $H_i$
and $X_i.$ In order to
apply the last one, we will need the following simple, elementary
lemma.

\begin{lemma}\label{l:selection} Let $c,\e>0$ be such  that $|(X+H)\cap \D|\ge c |X||H|^{1-\e}$. Then there is a set $X'\sbeq X$ such that $|X'|\ge \frac{c}4|X||H|^{-\e}$ and for every $Y\sbeq X'$ we have
$|(Y+H)\cap \D|\ge \frac{c^2}{8}|Y||H|^{1-2\e}.$
\end{lemma}
\begin{proof} Put $S=(X+H)\cap \D$ and notice that
$$\sum_{t\in X+H}(1_X*1_H)(t)=|X||H|\,.$$
Let us denote by $P$ the set of elements $t$ with $(1_X*1_H)(t)\ge \frac2{c}|H|^{\e}.$ Clearly, $|P|\le \frac{c}2|X||H|^{1-\e}$ and therefore
$$\sum_{t\in S\setminus P}(1_X*1_H)(t)=\sum_{x\in X}|(x+H)\cap (S\setminus P)|\ge |S|-|P|\ge \frac{c}2 |X||H|^{1-\e}\,.$$
Let $X'$ be the set of all $x\in X$ satisfying the inequality
 $|(x+H)\cap (S\setminus P)|\ge \frac{c}4
|H|^{1-\e}.$ Observe that
$$\sum_{x\in X\setminus X'}|(x+H)\cap (S\setminus P)|\le \frac{c}4|H|^{1-\e}|X\setminus X'|\le \frac{c}4|X||H|^{1-\e}\,,$$
hence
$$|X'||H|\ge \sum_{x\in X'}|(x+H)\cap (S\setminus P)|\ge \frac{c}4|X||H|^{1-\e}\,.$$
Thus, $|X'|\ge \frac{c}4 |X||H|^{-\e}$ and  if $Y\sbeq X'$, then
$$|(Y+H)\cap \D|\ge \frac{\frac{c}4|Y||H|^{1-\e}}{\frac2{c}|H|^\e}\,,$$
which yields to the required inequality. $\hfill\Box$
\end{proof}

\bigskip

 The next
lemma  provides a density increment in a simpler case-in
Bateman-Katz theorem we have $\a\ge 20f.$

\begin{lemma}\label{l:increment-3-easy}
    Let $A\subseteq \Z/N\Z, |A|=\d N,$ and assume that  for every $\D'\sbeq \D$ with $|\D'|\gg|\D|$ we have $E(\D')\gg |\D|^{7/3-2\mu}$ and  
		$E_8(\D)\le |\D|^{5+10\mu}.$   Then  either there is a regular Bohr set $B$  with $\rk( B)\ll \d^{-1+f}$  and radius $\Omega (\d^{1-f})$ such that
    $$|(A+t)\cap B|\gg \d^{1-f}|B|$$
    for some $t$; or there are sets $H$ and $X$ such that $|H|\ll |\D_\th|^{1/3+21f}, |H+H|\ll |H|^{1+f}$, $|X|\ll  |\D_\th|^{2/3+2f},$ and
    $$|(X+H)\cap \D_\th|\gg |\D_\th|^{1-21f}.$$
\end{lemma}
\begin{proof}  By Theorem \ref{t:bateman-katz-structure} applied with $\t=1/3-2\m$  there exist $0\le \a\le 1/3+\m$ and sets $H_i, X_i$ for 
$1\le i\le \lceil |\D_\th|^{\a-f}\rceil$ such that
$$|\D_\th|^{1/3-2\m+\a-2f}\ll |H_i|\ll |\D_\th|^{1/3-2\m+\a+f}\,,$$
and
$$|\D_\th|^{2/3+2\m-2\a-2f}\ll |X_i|\ll  |\D_\th|^{2/3+2\m-2\a+f}\,,$$
that fulfill inequalities \eqref{bk1}--\eqref{bk4}. 
First, we assume that $1/3-20f\le \a\le 1/3+\m$ and put  $k=\lceil |\D_\th|^{\a-25f}\rceil.$  Then by \eqref{bk4}, \eqref{bk3} and Theorem \ref{t:dimension} we have 
$$\big|\bigcup_{i=1}^k(X_i+H_i)\cap \D_\th\big|\gg |\D_\th|^{\a-25f}|\D_\th|^{1-\a-f}\ge |\D_\th|^{1-f}\ge \th^{-2}\d^{-f}\,,$$
and 
\begin{eqnarray*}
\dim\big(\bigcup_{i=1}^k(X_i+H_i)\big)&\ll& \sum_{i=1}^k|X_i|\dim (H_i)\le |\D_\th|^{\a-25f}|X_i||H_i|^f\log |H_i|\\
&\ll& |\D_\th|^{2/3+2\m-\a-22f}\le |\D_\th|^{1/3-f}\ll  \d^{-1+f}\,.
\end{eqnarray*}
 
 Next let us assume that $20f\le \a\le 1/3-20f$. Observe that by \eqref{bk4} for every $i$ we have
\begin{equation}\label{sel}
|(X_i+H_i)\cap \D_i|\gg |\D_\th|^{1-\a-f}\gg |X_i||H_i||\D_\th|^{-3f}\gg |X_i||H_i|^{1-5f}\,.
\end{equation}
By Lemma \ref{l:selection} applied with $X_i, H_i$ and $\e=5f$ there is   $X_i'\sbeq X_i$ such that
$$|X_i'|\gg |X_i||H_i|^{-5f}\gg |\D_\th|^{2/3+2\m-2\a-5f}\ge |\D_\th|^{1/3-\a+15f}\,.$$ 
Let  $Y_i\sbeq X_i'$ be any subset of size $\lceil|\D_\th|^{1/3-\a+15f}\rceil.$
    By Lemma \ref{l:selection}, we have
$$|(Y_i+H_i)\cap \D_i|\gg |\D_\th|^{2/3-2\m+3f}\ge |\D_\th|^{2/3+f}\ge \th^{-2}\d^{-f}\,.$$
Again  Theorem \ref{t:dimension} and \eqref{bk3} imply that
\begin{eqnarray*}
\dim(Y_i+H_i)&\le& \dim(Y_i)\dim(H_i)\ll |\D_\th|^{1/3-\a+15f}|H_i|^{f}\log |H_i|\\
&\le&  |\D_\th|^{1/3-\a+17f}\le |\D_\th|^{1/3-f}\le \d^{-1+f}\,.
\end{eqnarray*}

In both above considered cases we found a subset of $\D_\th$ of size $\Omega(\th^{-2}\d^{-f})$ and 
dimension $O(\d^{-1+f})$ hence  by Lemma \ref{l:l2-increment} there is a regular Bohr set
$B$ with $\rk( B)\ll \d^{-1+f}$  and  radius $\Omega (\d^{1-f})$ such that
$$|(A+t)\cap B|\ge (1+\Omega(\th^2\th^{-2}\d^{-f}))\d|B|\gg \d^{1-f}|B|$$
for some $t.$

Finally, if $\a\le 20f$ then for every $i$ we have  
$
|H_i|\ll
|\D_\th|^{1/3-\m/3+\a+f}\le |\D_\th|^{1/3+21f},
 |H_i+H_i|\ll
|H_i|^{1+f}, |X_i|\ll |\D_\th|^{2/3+2\m/3-2\a+f}\le
|\D_\th|^{2/3+2f} $
and
$$|(X_i+H_i)\cap \D_\th|\gg |\D_\th|^{1-f}\,,$$
which completes the proof.
$\hfill\Box$
\end{proof}

\bigskip

Finally, we arrived at a more difficult case, where   $\D\approx
X+H, \, |X|\sim\d^{-2+O(\m)}, |H|\sim\d^{-1+O(\m)}$ and the set $H$ is highly
structured.

\begin{lemma}\label{l:increment-3-hard}  Let $A\subseteq \Z/N\Z, |A|=\d N,$ and assume that there are sets $H$ and $X$ such that $|H|\ll |\D_\th|^{1/3+21f}, |H+H|\ll |H|^{1+f}$, $|X|\ll  |\D_\th|^{2/3+2f},$ and
    $$|(X+H)\cap \D_\th|\gg |\D_\th|^{1-21f}.$$  Then   there is a regular Bohr set $B$  with $\rk( B)\le \d^{-1+f}$  and  radius $\Omega (\d^{1-f})$ such that
    $$|(A+t)\cap B|\gg \d^{1-f}|B|$$
    for some $t$.
\end{lemma}
\begin{proof}  Since $\dim(H)\ll \d^{-2f}$ then there is a set $\La$ such that $|\La|\ll \d^{-2f}$ and $H\sbeq \Span(\La)$. Let $B=B(\Lambda,\gamma)$ be  a regular Bohr set with radius
$1/(6|\Lambda|)\le \gamma\le 1/(3|\Lambda|)$ (the existence of such
$\g$ is guaranteed by Lemma \ref{l:bohr-regular}). Then clearly
$B\sbeq B(H,1/3)$. Furthermore, for $h\in H$ and $b\in B$ we have
$$\|hb/N\|\le \sum_{\lambda\in \La} \|\lambda b/N\|\le 1/3,$$
so
$$|\hB(h)|\ge \sum_{b\in B}\Re \,e^{-2\pi ihb/N}\ge \frac12|B|\,.$$
 Put $A_t=(A+t)\cap B$ and  let us assume that for each $t$ we have
\begin{equation}\label{in:density-at}
|A_t|\ll \d^{1-f} |B|\,,
\end{equation}
as otherwise we would obtain the required density increment on a Bohr set with the rank
$O(\d^{-2f})$ and  radius $\Omega(\d^{2f})$.
For every $x\in \zn$ we have
$$\h {1_{A_t}}(x)=\frac1N\sum_h\hB(h)\hA(x-h)e^{2\pi i t(x-h)/N}\,,$$
hence by the Parseval formula
 \begin{equation}\label{for}
\sum_t|\h {1_{A_t}}(x)|^2=\frac1N\sum_h |\hB(h)\hA(x+h)|^2\ge \frac1N\sum_{h\in H} |\hB(h)\hA(x+h)|^2\,.
\end{equation}
Let $Y\sbeq X$ be a set given by Lemma \ref{l:selection} when applied to $X, H$ and $\e=5f$. Bounding similarly as in  \eqref{sel} we have 
$$|Y|\gg |X||H|^{-5f}\gg |\D_\th|^{2/3-O(f)}\,.$$ 
Therefore,  summing \eqref{for} over $Y$ we get
\begin{eqnarray*}\sum_{t}\sum_{x\in Y}|\h {1_{A_t}}(x)|^2&\ge& \frac1N\sum_{h\in H}\sum_{x\in Y} |\hB(h)\hA(x+h)|^2 \gg \frac1N |(Y+H)\cap \D_\th||B|^2 \d^{2-2\m}|A|^2\\
&\gg &\d^3 |\D_\th|^{1-O(f)}|B|^2|A|\gg \d^{O(f)}|B|^2|A|\,.
\end{eqnarray*}
Using averaging argument and  (\ref{in:density-at}) we see that there is a $t$ such that
\begin{equation}\label{sum}
\sum_{x\in Y}|\h {1_{A_t}}(x)|^2\gg  \d^{1+O(f)}|B|^2\gg \d^{-1+O(f)}|A_t|^2\,.
\end{equation}
We can ignore  all small terms in \eqref{sum} that satisfy
$$|\h {1_{A_t}}(x)|\le  c\frac{\d^{-1/2+O(f)}}{|X|^{1/2}}|A_t|=\d^{1/2+O(f)}|A_t| \,,$$
where $c>0$ is a sufficiently small constant,
so by  dyadic argument there is $\eta\gg \d^{1/2+O(f)}$ such that
$$\sum_{x\in Y: \,\eta |A_t|\le |\h {1_{A_t}}(x)|\le 2\eta |A_t|}|\h {1_{A_t}}(x)|^2\gg  \d^{-1+O(f)}\log^{-1}(1/\d)|A_t|^2\gg \d^{-1+O(f)}|A_t|^2\,.$$
Put 
$$S=\big\{x\in Y: \eta |A_t|\le |\h {1_{A_t}}(x)|\le 2\eta |A_t|\big\}$$ 
then by the above inequality it follows that
$$|S|\gg  \eta^{-2}  \d^{-1+O(f)}\,.$$
 By Lemma \ref{l:bloom}
there is a set $Z\sbeq S\sbeq Y$ such that $|Z|\ge \eta |S|\gg \eta^{-1} \d^{-1+O(f)}$ and  $\dim(Z)\ll \eta^{-1}\log(N/|A_t|).$
From \eqref{sum} one can deduce that $|A_t|\gg  \d^2|B|$ hence by Lemma \ref{l:bohr-size} it follows that
$$|A_t|\gg  \d^2|B|\ge \d^2 \g^{\d^{-2f}}N\ge (\d/8)^{2\d^{-2f}}N\,,$$
so 
$$\dim(Z)\ll \eta^{-1}\d^{-3f}\ll \d^{-1/2-O(f)}\,.$$

Put $\eta_1=\eta$ and let $Z_1\sbeq Y$ be any set of size $\Theta(\eta_1^{-1}\d^{-1+O(f)})$ such that $\dim(Z)\ll \eta_1^{-1}\d^{-3f}\ll \d^{-1/2-O(f)}$.  Then we apply the above argument to the  set $Y\setminus Z_1$ to find $Z_2\sbeq Y\setminus Z_1$ and  $\eta_2\gg \d^{1/2+O(f)}$ with the same properties
(observe that the whole argument can be applied for any set $Y'\sbeq Y$ giving essentially the same conclusion, as long as
$|Y'|\gg |Y|$). Applying this procedure $k$ times we obtain disjoint sets $Z_1,\dots,Z_k\sbeq Y$ such that $|Z_i|=\Theta(\eta_i^{-1}\d^{-1+O(f)})$ and $\dim(Z_i)\ll
\eta_i^{-1}\d^{-2f}\ll \d^{1-O(f)}|Z_i|$ for some
$\eta_i\gg \d^{1/2+O(f)},$ where $k$ is the smallest integer such that
$$|Z_1|+\dots+|Z_k|\ge \d^{-3/2}\,.$$
Since for each $i$ we have $|Z_i|\le \d^{-3/2}$ it follows that
$$|Z_1|+\dots+|Z_k|\le 2\d^{-3/2}\,.$$  
Put $U=\bigcup_{i=1}^k Z_i\sbeq X$ then $|U|\ge \d^{-3/2}$ and
$$\dim(U+H)\le\dim(H)\sum_{i=1}^k\dim(Z_i)\ll \d^{1-O(f)}\sum_{i=1}^k |Z_i| \ll \d^{-1/2-O(f)}\,.$$
 Again, by Lemma \ref{l:selection} we have
$$|U+H|\gg |U||H|^{1-10f}\gg  \d^{-5/2+O(f)}\,.$$
Lemma \ref{l:l2-increment} implies that there exists a Bohr  set $B'$ with $\rk(B')\ll \d^{-1/2-O(f)}$ and radius $\Omega(\d^{1/2+O(f)})$ such that
\begin{equation}\label{f}
|(A+t)\cap B'|\gg (1+\Omega(\d^{2+2\m} \d^{-5/2+O(f)}))\d |B'|\gg \d^{1/2+O(f)}|B'|\gg \d^{1-f}|B'|
\end{equation}
for some $t$ which is a contradiction. $\hfill\Box$
\end{proof}

\section{Large Fourier coefficients}
\label{s:large-coeff}
In this section we obtain the density increment if \eqref{mid} and \eqref{sml} do not hold, so there is a kind of  spectral gap in terms of $L^3$-norm. 
We will use the well-known Chang's Spectral Lemma \cite{chang}, which states that for every $\th$ we have
$$\dim(\D_\th)\ll \th^{-2}\log (1/\d)\,.$$
For any function $f: \Z_N\rightarrow \mathbb R$ define 
$$T(f)=\sum_{x+y=2z}f(x)f(y)f(2z)\,.$$  We also make use of the following  lower bound on the number of $3$-term arithmetic progressions in a set $S\sbeq \zn$ with density $\g$ proven by Bloom \cite{bloom}
$$T(1_A)\ge \g^{O(\g^{-1}\log^4(1/\g))}N^2.$$

\begin{lemma}\label{l:increment-4}  Let $A\subseteq \Z/N\Z,$ be a set with density $\d$ such that \eqref{mid} and \eqref{sml} do not hold. Then   there is a regular 
Bohr set $B$  
with $\rk( B)\le \d^{-2/5}$  and radius $\Omega (\d^{4})$ such that
    $$|(A+t)\cap B|\gg \m\frac{\log(1/\d)}{\log\log^5(1/\d)}\d|B|$$
    for some $t$.
\end{lemma}
\begin{proof} By Chang's lemma 
$$\dim(\D_{\d^{1/10}})\ll \d^{-1/5}\log (1/\d)\le \d^{-2/5}$$
 hence there is a set $\La$ such that $|\La|\ll  \d^{-2/5}$ and 
$\D_{\d^{1/10}}\sbeq \Span(\La)$. Let $B=B(\Lambda,\g)$ be  a regular Bohr set with radius $\g\gg \d^3$. Let $\b=\frac1{|B|}1_B$ then for every $r\in \D_{\d^{1/10}}$ we have
\begin{equation}\label{beta}
\big|\h \b(r)-1\big|\le \frac1{|B|}\sum_{b\in B} |e^{-2\pi i\lambda b/N}-1|\le \frac{2\pi}{|B|}\sum_{b\in B} \sum_{\lambda\in \La}\|rb/N\|\le 2\pi \d^2\,,
\end{equation}
and similarly $|\h \b(2r)-1| \ll \d^2.$
Let $f: \Z_N\rightarrow [0,1]$ be a function defined by
$$f(t)=\b*1_A(t)\,.$$
 We may assume that  $f(t)=\frac1{|B|}|(A+t)\cap B|\le L\d,$
where
$$L=c\m\frac{\log(1/\d)}{\log\log^5(1/\d)}$$ 
and $c>0$ is a small constant. Put
$$S=\big\{ t: f(t)\ge \d/2\big \}$$
then by $\sum_tf(t)=|A|$ it follows that $|S|\ge N/L$ hence by Bloom's Theorem we have
\begin{equation}\label{tf}
T(f)\ge \frac18\d^3T(S)\gg \d^3 \exp({-O(L\log^5 L)})N^2\gg \d^{3+\m/10}N^2\,.
\end{equation}
Our next step is to compare $T(f)$ and $T(1_A).$ By \eqref{beta}, \eqref{mid}, \eqref{sml}, Parseval's formula and  H\"older's inequality  we have
\begin{eqnarray*}
\big|T(1_A)-T(f)\big|&=&\frac1{N}\big|\sum_{r=0}^{N-1} \h {1_A}(r)^2\hA(-2r)-\sum_{r=0}^{N-1} \h f(r)^2\h f(-2r)\big|\\
&\le&\frac1{N}\sum_{r\in \D_{\d^{1/10}}} |\hA(r)^2\hA(-2r)(1-\h \b(r)^2\h \b (-2r))|+\frac2{N}\sum_{r\not\in \D_{\d^{1/10}}}|\hA(r)|^3\\
&\ll & \d^2 \frac1{N}\sum_{r\in \D_{\d^{1/10}}}|\hA(r)|^3+\frac2{N}\sum_{r\in \D_{\d^{1+\m}}\setminus \D_{\d^{1/10}} }|\hA(r)|^3+\frac2{N}\sum_{r\not\in \D_{\d^{1+\m}}}|\hA(r)|^3\\
&\ll& \d^2|A|^2+\d^{1+\m/5}|A|^2+2\d^{1+\m}|A|^2\ll \d^{3+\m/5}N^2\,.
\end{eqnarray*}
Thus, by \eqref{tf} 
$$T(1_A)\gg \d^{3+\m/10}N^2\,,$$
which is a contradiction.$\hfill\Box$
\end{proof}

\section{Proof of Theorem 1}\label{s:iteration}

Summarizing all considered cases, we can state the following result.

\begin{theorem}\label{t:increment}
There exists an absolute constant $c>0$ such that   the following holds.
 Let $A\subseteq \Z/N\Z$ be a set without any non-trivial arithmetic progressions of length three and let $|A|=\d N$.  Then  there is a regular Bohr set $B$ with 
$\rk( B)\ll \d^{-1+c}$  and  radius $\Omega( \d^{4})$ such that     for some $t$
    $$|(A+t)\cap B|\gg \frac{\log(1/\d)}{\log\log^5(1/\d)}\d|B|.$$

\end{theorem}
\begin{proof}
Let us first make a suitable  choice of parameters . Let  $\m>0$ be a constant
the such that  for every $\s\le \mu,$  \eqref{f} holds with  $f=f(16\s)$.  Since
we assumed that $f(\m)\ge \m,$ we see that in all considered cases
in Lemma \ref{l:increment-1}, Lemma \ref{l:increment-2},
Lemma \ref{l:increment-3-easy},  Lemma \ref{l:increment-3-hard} and Lemma \ref{l:increment-4} we obtain  density increment at least by  factor
of $\Omega(\m\log(1/\d)\log\log^{-5}(1/\d))$ on a Bohr set with $\rk( B)\ll \d^{-1+\m/3}$  and
 radius $\Omega( \d^{4})$. Thus, it is enough to take $c=\m/3$. $\hfill\Box$
\end{proof}

\bigskip

After the first step of our iterative procedure we obtain a larger density increment on a
low-rank Bohr set and then   we  apply less effective method of Bloom (Theorem 7.1 \cite{bloom}).

\begin{lemma}\label{l:bloom-iteration} {\rm\cite{bloom}} There exists an absolute constant $c_1 > 0$ such that the following
holds. Let $B\sbeq \Z/N\Z$ be a regular Bohr set of rank $d$. Let
$A_1 \sbeq B$ and $A_2 \sbeq B_\e,$ each with relative densities
$\a_i$. Let $\a = \min(c_1, \a_1, \a_2)$ and assume that   $d \le
\exp(c_1(\log^2(1/\a)).$ Suppose that   $B_\e$ is also regular and
$c_1\a/(4d) \le \e \le c_1\a/d.$ Then either
\begin{itemize}
\item[(i)] there is a regular Bohr set $B'$ of rank $\rk(B') \le d +
O(\a^{-1}\log(1/\a))$ and size
$$ |B'| \ge \exp\big(-O(\log^2(1/\a)(d + \a^{-1}\log(1/\a)))\big)|B|$$
such that 
$$|(A_1+t)\cap B'|\gg (1+c_1)\a_1|B'|$$
for some $t\in \zn$;
\item[(ii)] or there are $\Omega(\a_1^2\a_2 |B| |B_\e|)$ three-term arithmetic progressions $x+y=2z$ with $x,y\in A_1, z\in A_2$;
\end{itemize}
\end{lemma}

Now we are in position to finish the proof of our main result. We will not give detailed proof of the iteration procedure as it is very standard and the reader can find details on it in the literature (see \cite{bloom}, \cite{sanders-3/4}).  In the first step we apply Theorem \ref{t:increment} to obtain  a regular Bohr set $B^0$
with $\rk( B^0)\ll \d^{-1+c}$, radius $\Omega( \d^{4})$ and  a  progression-free  set $A_0\sbeq A+t$ for some $t$ such that 
$$|A_0\cap B^0|\gg \a|B^0|\,,$$
where 
$$\a\gg \frac{\log(1/\d)}{\log\log^5(1/\d)}\d\,.$$
By Lemma \ref{l:bohr-size} we have 
$$|B^0|\ge \exp\big (-O(\d^{-1+c}\log(1/\d))\big)N\,.$$
Next we 
iteratively  apply  Lemma \ref{l:bloom-iteration}   and let $B^i$ be Bohr sets obtained in the iterative procedure. Observe  that after $k\ll \log (1/\a)$ steps we will be in the case $(ii)$ of Lemma \ref{l:bloom-iteration} and that  $\rk(B^i)\ll \a^{-1}\log^2(1/\a)$ for every $i\le k$. Thus, there are
$$\Omega( \a^3 |B^k||B^k_\e|)$$
three-term arithmetic progressions in $A,$ where $\e\ge c_1\a/(4\rk(B^k))\gg \a^2\log^2(1/\a).$ 

 Hence by Lemma \ref{l:bloom-iteration}
we have 
$$|B^k|\ge \exp\big(-O(\a^{-1}\log^4(1/\a))\big)N\ge  \exp\big(-O(\d^{-1}\log^3(1/\d))\log\log^5(1/\d)\big)N\,,$$
and by Lemma \ref{l:bohr-size}
\begin{eqnarray*}
|B^k_\e|&\ge& \exp\big(-O(\a^{-1}\log^3(1/\a))\big)\exp\big(-O(\a^{-1}\log^4(1/\a))\big)N\\
&\ge&  \exp\big(-O(\d^{-1}\log^3(1/\d)\log\log^5(1/\d))\big)N\,.
\end{eqnarray*}
Therefore $A$ contains 
 $$\a^3\exp\big(-O(\d^{-1}\log^3(1/\d)\log\log^5(1/\d))\big)N^2$$
arithmetic progressions of length three.
Since there are only $|A|$ trivial progressions it follows that
$$|A|\ge \a^3 \exp\big(-O(\d^{-1}\log^3(1/\d)\log\log^5(1/\d))\big)N^2\,,$$
which completes the proof of Theorem \ref{t:roth}.

\section{Concluding remarks}

In Lemma \ref{l:increment-1}, Lemma \ref{l:increment-2},
Lemma \ref{l:increment-3-easy} and   Lemma \ref{l:increment-3-hard} we obtained a density increment by factor of $\d^{-c}$ on a low-rank Bohr set, where $c$ is a positive constant. Such  density increment even in the first step of an iterative method would lead to the upper bound $O((\log N)^{-1-c})$ in Theorem \ref{t:roth}. However, in Lemma \ref{l:increment-4} we only were able to prove  an increment by factor $(\log(1/\d))^{1-o(1)}$. Any refinement of   Lemma \ref{l:increment-4} will directly imply an improvement of  Theorem \ref{t:roth}.
{}

\bigskip

\no{Faculty of Mathematics and Computer Science,\\ Adam Mickiewicz
University,\\ Umul\-towska 87, 61-614 Pozna\'n, Poland\\} {\tt
schoen@amu.edu.pl}

\end{document}